\documentclass[12pt]{article}

\usepackage[utf8]{inputenc}
\usepackage{amssymb}
\usepackage{amsmath}
\usepackage{amsfonts}
\usepackage{indentfirst}
\usepackage{tikz}
\usepackage{float}
\usepackage{eucal}
\usepackage{amsthm}
\usepackage{amstext}
\usepackage{tabularx}
\usepackage{graphicx}
\usepackage{xcolor}
\usepackage{setspace}

\DeclareGraphicsExtensions{.eps,.bmp,.jpg,.pdf,.mps,.png,.gif}

\newtheorem{theorem}{Theorem}[section]

\newtheorem{case}{Case}[section]
\newtheorem{claim}{Claim}[section]

\newtheorem{definition}{Definition}[section]

\newtheorem{lemma}{Lemma}[section]
\newtheorem{notation}{Notation}[section]

\newtheorem{proposition}{Proposition}[section]
\newtheorem{remark}{Remark}[section]

\begin{document}

\begin{singlespace}

\begin{center}
\bigskip

{\Large Orbital Fuzzy Iterated Function Systems}

\bigskip

Alexandru MIHAIL, Irina SAVU

\bigskip
\end{center}

\begin{quotation}
\textbf{Abstract}: In this paper we introduce the concept of orbital fuzzy
iterated function system and prove that the fuzzy operator associated to
such a system is weakly Picard. An example is provided.

\textbf{Keywords}: orbital fuzzy iterated function system, fuzzy operator,
attractor.

\textbf{Mathematics Subject Classification (2010): \ }28A80, 37C70, 54H25.
\end{quotation}

\bigskip

\section{Introduction}

\bigskip

The concept of fractal was introduced by B. Mandelbrot (see \cite{Mandelbrot}%
), as a set whose Hausdorff dimension strictly exceeds its topological
dimension. He also indicated that self-similarity (a property that is
present in the case of the classical fractal sets as Cantor's ternary set,
Sierpinski triangle etc.) is very important in the study of such sets. J.
Hutchinson (see \cite{Hutchinson}) was the one to set up a theory of self
similar sets in a general framework, by introducing iterated function
systems (for short IFSs), even though some other mathematicians, as P. Moran
(see \cite{Moran}) and R. Williams (see \cite{Williams}), worked in this
direction before him.

Iterated function systems have applications in several mathematical fields.
Among them we mention: statistical mechanics (see \cite{Stenflo}), Monte
Carlo algorithms (see \cite{Roberts}), random matrices (see \cite{Kaij}),
continued fractions (see \cite{Iosifescu2}), wavelets (see \cite{Jorg}), etc.

Three excellent surveys of IFSs are \cite{Diaconis}, \cite{Iosifescu} and
\cite{Lesn}.

Many mathematicians provided generalizations of Hutchinson's results by
allowing weaker forms of contractions: average contractions (see \cite%
{Barnsley}), $\varphi $-contractions (see \cite{Georgescu}, \cite{Georgescu2}%
, \cite{Ioana}, \cite{Jach} and \cite{Mate}), Hardy-Rogers contractions (see
\cite{Georgescu3}),\ \'{C}iri\'{c} contractions (see \cite{Petr}), Reich
contractions (see \cite{M1}), Kannan and Chatterjea contractions (see \cite%
{VanD}), convex contractions (see \cite{M2}), etc.

On the one hand, in the view of the topic treated in this paper, a special
place in the above mentioned list is occupied by the iterated function
systems consisting of continuous functions satisfying Banach's orbital
condition (see \cite{M3}, \cite{Savu} and \cite{MS}). We emphasize that the
fractal operator associated to such a system is weakly Picard, in contrast
to the previous mentioned generalizations.

On the other hand, another kind of generalization of Hutchinson's concept of
iterated function system, namely the fuzzy version of it, was initiated by
Cabrelli, Forte, Molter, and Vrscay (see \cite{Cabrelli}) by considering the
attractor of an IFS as a certain fuzzy set. For some other results along
this line of research see \cite{AR}, \cite{AR2}, \cite{Cunha}, \cite{Forte}
and \cite{Str}).

In this paper, we combine the two afore-mentioned research directions by
introducing the concept of orbital fuzzy iterated function system (see
Definition \ref{dfuzzy}) and proving that its fuzzy Hutchinson-Barnsley
operator is weakly Picard (see Theorem \ref{thm}). Finally, an illustrative
example is presented.

\bigskip

\section{Preliminaries}

\bigskip

Given a set $X$, a function $f:X\rightarrow X$ and $n\in
\mathbb{N}
,$ by $f^{n}$ we mean $f\circ f\circ ...\circ f$ \ by $n$ times.

Given a metric space $\left( X,d\right) ,$ by:

- a weakly Picard operator we mean a function $f:X\rightarrow X$ having the
property that, for every $x\in X,$ the sequence $\left( f^{n}\left( x\right)
\right) _{n\in
\mathbb{N}
}$ is convergent to a fixed point of $f$;

- $diam\left( A\right) $ we mean the diameter of the subset $A$ of $X$,
defined by%
\begin{equation*}
diam\left( A\right) =\sup_{x,y\in A}d\left( x,y\right) ;
\end{equation*}

- $P_{b}\left( X\right) $ we mean the set of non-empty bounded subsets of $%
X; $

- $P_{cl,b}\left( X\right) $ we mean the set of non-empty closed and bounded
subsets of $X;$

- $P_{cp}\left( X\right) $ we mean the set of non-empty compact subsets of $%
X;$

- the Hausdorff-Pompeiu semimetric we mean $h:P_{b}\left( X\right) \times
P_{b}\left( X\right) \rightarrow \lbrack 0,+\infty )$ given by%
\begin{equation*}
h\left( A,B\right) =\max \left\{ d\left( A,B\right) ,d\left( B,A\right)
\right\}
\end{equation*}%
for all $A,B\in P_{b}\left( X\right) ,$ where $d\left( A,B\right)
=\sup_{x\in A}\inf_{y\in B}d\left( x,y\right) $ and $d\left( B,A\right)
=\sup_{x\in B}\inf_{x\in A}d\left( x,y\right) $. The restriction of $h$ to $%
P_{cl,b}\left( X\right) $ is called the Hausdorff-Pompeiu metric and it is
also denoted by $h$.

\begin{lemma}
\label{conv}Let $\left( X,d\right) $ be a metric space. Let us consider $%
f:X\rightarrow X$ a continuous function and $\left( f_{n}\right) _{n}$ a
sequence with $f_{n}:X\rightarrow X$ continuous for all $n\in
\mathbb{N}
$, such that $\left( f_{n}\right) _{n}$ is uniform convergent to $f$. Let $%
\left( x_{n}\right) _{n}\subset X$ be a sequence that is convergent to $x\in
X$. Then,%
\begin{equation*}
\lim_{n\rightarrow \infty }f_{n}\left( x_{n}\right) =f\left( x\right) \text{.%
}
\end{equation*}
\end{lemma}

\begin{lemma}
\label{lema_maxmax}Let $I$ be a finite set, $J$ an infinite set and $\left(
a_{ij}\right) _{i\in I\text{,}j\in J}$ a family of elements from $%
\mathbb{R}
$, such that $\sup_{j\in J}a_{ij}=\max_{j\in J}a_{ij}$. Then,%
\begin{equation*}
\max_{i\in I}\max_{j\in J}a_{ij}=\max_{j\in J}\max_{i\in I}a_{ij}\text{.}
\end{equation*}
\end{lemma}

\bigskip

\textbf{Results regarding the Hausdorff-Pompeiu semimetric}

\begin{proposition}
\bigskip \lbrack see \cite{Sec2}] \ For a metric space $\left( X,d\right) $,
we have:%
\begin{equation}
h\left( \cup _{i\in I}A_{i},\cup _{i\in I}B_{i}\right) \leq \sup_{i\in
I}h\left( A_{i},B_{i}\right)  \label{sup}
\end{equation}%
for every $\left( A_{i}\right) _{i\in I\text{ }}$and $\left( B_{i}\right)
_{i\in I}$ families of elements from $P_{b}\left( X\right) .$
\end{proposition}

\begin{proposition}
\label{diam}\bigskip \lbrack see \cite{Sec2}] For a metric space $\left(
X,d\right) $, we have%
\begin{equation*}
h\left( A,B\right) \leq diam\left( A\cup B\right)
\end{equation*}%
for every $A,B\in P_{b}\left( X\right) $.
\end{proposition}

\begin{proposition}
\label{h_cont}Let $\left( X,d\right) $ be a metric space, $\left(
K_{n}\right) _{n\in
\mathbb{N}
}\subseteq P_{cp}\left( X\right) $ and $K\in P_{cp}\left( X\right) $ such
that $\lim_{n\rightarrow \infty }h\left( K_{n},K\right) =0.$ Let $%
f:X\rightarrow X$ be a continuous function. Then, $\lim_{n\rightarrow \infty
}h\left( f\left( K_{n}\right) ,f\left( K\right) \right) =0$.
\end{proposition}

\begin{proposition}
\label{lim_comp} Let $\left( X,d\right) $ be a complete metric space, $%
\left( K_{n}\right) _{n\geq 0}\subset P_{cp}\left( X\right) $ and $K\ $a
closed subset of $X$ such that $\lim_{n\rightarrow \infty }h\left(
K_{n},K\right) =0$. Then, $K\in P_{cp}\left( X\right) $.
\end{proposition}

\begin{proposition}
\lbrack see Proposition 2.8 from \cite{M2}] \label{reun_comp}Let $\left(
X,d\right) $ be a complete metric space, $\left( K_{n}\right) _{n\geq
0}\subset P_{cp}\left( X\right) $ and $K\in P_{cp}\left( X\right) $ such
that $\lim_{n\rightarrow \infty }h\left( K_{n},K\right) =0$. Then, $K\cup
\left( \cup _{n=0}^{\infty }K_{n}\right) \in P_{cp}\left( X\right) $.
\end{proposition}

\begin{proposition}
\lbrack see \cite{Sec2}] \ If the metric space $\left( X,d\right) $ is
complete, then the metric space $\left( P_{cp}\left( X\right) ,h\right) $ is
complete.
\end{proposition}

\bigskip

\textbf{Orbital iterated function systems}

\bigskip

\begin{definition}
Let $\left( X,d\right) $ be a complete metric space and $\left( f_{i}\right)
_{i\in I}$ a finite family of continuous functions, where $%
f_{i}:X\rightarrow X$ for all $i\in I$. Let $B\in P_{cp}\left( X\right) $.
By the orbit of $B$ we mean the set $\mathcal{O}\left( B\right) =\cup _{n\in
\mathbb{N}
}\cup _{\alpha \in \Lambda \left( I\right) }f_{\left[ \alpha \right]
_{n}}\left( B\right) $. If $B=\left\{ x\right\} $, for the orbit of $\left\{
x\right\} $ we make the notation $\mathcal{O}\left( x\right) $.
\end{definition}

\begin{definition}
Let $\left( X,d\right) $ be a complete metric space. An orbital contractive
iterated function system is a finite family of continuous functions $\left(
f_{i}\right) _{i\in I},$ where $f_{i}:X\rightarrow X$, having the property
that there exists $C\in \lbrack 0,1)$ such that%
\begin{equation*}
d\left( f_{i}\left( y\right) ,f_{i}\left( z\right) \right) \leq Cd\left(
y,z\right)
\end{equation*}%
for every $x\in X,$ $i\in I$ and $y,z\in \mathcal{O}\left( x\right) .$ We
denote such a system by $\mathcal{S}=\left( \left( X,d\right) ,\left(
f_{i}\right) _{i\in I}\right) .$
\end{definition}

\begin{definition}
The fractal operator associated to the orbital contractive iterated function
system $\mathcal{S}=\left( \left( X,d\right) ,\left( f_{i}\right) _{i\in
I}\right) $ is the function $F_{\mathcal{S}}:P_{cp}\left( X\right)
\rightarrow P_{cp}\left( X\right) $ defined by
\begin{equation*}
F_{\mathcal{S}}\left( K\right) =\cup _{i\in I}f_{i}\left( K\right)
\end{equation*}%
for every $K\in P_{cp}\left( X\right) .$
\end{definition}

\begin{definition}
Every fixed point of the fractal operator associated to an orbital
contractive iterated function system is called an attractor of the system.
\end{definition}

\begin{theorem}
\lbrack see \cite{MS}] Every orbital contractive iterated function system
has attractor. More precisely, the fractal operator associated to such a
system is a weakly Picard operator.
\end{theorem}

\begin{notation}
Let $\mathcal{S}=\left( \left( X,d\right) ,\left( f_{i}\right) _{i\in
I}\right) $ be an orbital contractive iterated function system. Then for
every $K\in P_{cp}\left( X\right) $, there exists a set (an attractor)
denoted by $A_{K}\in P_{cp}\left( X\right) $ such that $\lim_{n\rightarrow
\infty }h\left( F_{S}^{n}\left( K\right) ,A_{K}\right) =0$. If $K=\left\{
x\right\} $, we will denote its attractor by $A_{x}$ instead of $A_{\left\{
x\right\} }$.
\end{notation}

\begin{remark}
Let $\mathcal{S}=\left( \left( X,d\right) ,\left( f_{i}\right) _{i\in
I}\right) $ be an orbital contractive iterated function system. As for every
$x\in X$ the sequence $\left( F_{\mathcal{S}}^{n}\left( \left\{ x\right\}
\right) \right) _{n}$ is convergent to $A_{x}$, we have $\overline{\mathcal{O%
}\left( x\right) }=\mathcal{O}\left( x\right) \cup A_{x}$.
\end{remark}

Using a technique similar with the one used in \cite{Savu}, one can prove
the following:

\begin{proposition}
\label{Ak}Let $\mathcal{S}=\left( \left( X,d\right) ,\left( f_{i}\right)
_{i\in I}\right) $ be an orbital contractive iterated function system. If $\
$the sequence $\left( K_{n}\right) _{n}$ with $K_{n}\in P_{cp}\left(
X\right) $ converges to an element $K\in P_{cp}\left( X\right) ,$ then $%
\left( A_{K_{n}}\right) _{n}$ converges to $A_{K}.$
\end{proposition}

\begin{lemma}
\label{convpeorb}Let $\mathcal{S}=\left( \left( X,d\right) ,\left(
f_{i}\right) _{i\in I}\right) $ be an orbital contractive iterated function
system. Let $\left( t_{n}\right) _{n\in
\mathbb{N}
}$ and $\left( z_{n}\right) _{n\in
\mathbb{N}
}$ be two sequences with elements from $X$ such that $\displaystyle%
\lim_{n\rightarrow \infty }t_{n}=t$, $\displaystyle\lim_{n\rightarrow \infty
}z_{n}=z$ with $t,z\in X$ and $z_{n}\in \overline{\mathcal{O}\left(
t_{n}\right) }$ for all $n\in
\mathbb{N}
$. Then, $z\in \overline{\mathcal{O}\left( t\right) }$.
\end{lemma}

\begin{proof}
We have that for all $K\in P_{cp}\left( X\right) $, there exists a set $%
A_{K}\in P_{cp}\left( X\right) $ such that $\displaystyle\lim_{n\rightarrow
\infty }F_{\mathcal{S}}^{n}\left( K\right) =A_{K}$. Passing to a
subsequence, we distinguish the following cases:

\begin{case}
There exists $m\in
\mathbb{N}
$ such that $z_{n}\in F_{\mathcal{S}}^{m}\left( \left\{ t_{n}\right\}
\right) $ for all $n\in
\mathbb{N}
$.
\end{case}

As $F_{\mathcal{S}}$ is uniformly continuous on compact sets and $\left(
t_{n}\right) _{n\in
\mathbb{N}
}$ converges to $t$, applying Proposition \ref{h_cont} we have%
\begin{equation}
\lim_{n\rightarrow \infty }F_{\mathcal{S}}^{m}\left( \left\{ t_{n}\right\}
\right) =F_{\mathcal{S}}^{m}\left( \left\{ t\right\} \right) \text{.}
\label{0}
\end{equation}

As $z_{n}\in F_{\mathcal{S}}^{m}\left( \left\{ t_{n}\right\} \right) $ and $%
\left( z_{n}\right) _{n}$ is convergent to $z$, applying (\ref{0}) we obtain
$z\in \overline{F_{\mathcal{S}}^{m}\left( \left\{ t\right\} \right) }=F_{%
\mathcal{S}}^{m}\left( \left\{ t\right\} \right) $ and we conclude that $%
z\in \overline{\mathcal{O}\left( t\right) }$.

\begin{case}
\begin{equation*}
\lim_{n\rightarrow \infty }h\left( \left\{ z_{n}\right\} ,A_{t_{n}}\right)
=0.
\end{equation*}
\end{case}

We deduce there is a sequence $\left( s_{n}\right) _{n\in
\mathbb{N}
}$ such that $s_{n}\in A_{t_{n}}$ for all $n\in
\mathbb{N}
$ and $\displaystyle d\left( z_{n},s_{n}\right) \leq h\left( \left\{
z_{n}\right\} ,A_{t_{n}}\right) +\frac{1}{n}$. As $\displaystyle%
\lim_{n\rightarrow \infty }h\left( \left\{ z_{n}\right\} ,A_{t_{n}}\right)
=0 $ and $\displaystyle\lim_{n\rightarrow \infty }z_{n}=z$ we obtain%
\begin{equation}
\lim_{n\rightarrow \infty }s_{n}=z\text{.}  \label{01}
\end{equation}

Using Proposition \ref{Ak} and the fact that $\lim_{n\rightarrow \infty
}t_{n}=t$ we deduce%
\begin{equation}
\lim_{n\rightarrow \infty }h\left( A_{t_{n}},A_{t}\right) =0\text{.}
\label{02}
\end{equation}

Applying (\ref{01}), (\ref{02}) and taking into account that $s_{n}\in
A_{t_{n}}$ for all $n\in
\mathbb{N}
$ we conclude that $z\in A_{t}\subset \overline{\mathcal{O}\left( t\right) }$%
.
\end{proof}

\bigskip

\textbf{The shift (code) space for an iterated function system}

\bigskip

$%
\mathbb{N}
$ denotes the natural numbers, $%
\mathbb{N}
^{\ast }=%
\mathbb{N}
$ $\diagdown \left\{ 0\right\} $ and $%
\mathbb{N}
_{n}^{\ast }=\left\{ 1,2,...,n\right\} $, where $n\in
\mathbb{N}
^{\ast }$. Given two sets $A$ and $B$, by $B^{A}$ we mean the set of all
functions from $A$ to $B$.

For a set $I,$ we use the notation $I^{%
\mathbb{N}
_{n}^{\ast }}\overset{not}{=}\Lambda _{n}\left( I\right) $, hence the
elements of $\Lambda _{n}\left( I\right) $ can be written as words $\omega
=\omega _{1}\omega _{2}...\omega _{n}$ with $n$ letters from $I$, and in
this case $n$ is called the length of $\omega $ and it is denoted by $%
|\omega |.$ We also use the notation $\cup _{n\in
\mathbb{N}
}\Lambda _{n}\left( I\right) =\Lambda ^{\ast }\left( I\right) $, where $%
\Lambda _{0}\left( I\right) $ consists of a single element, namely the empty
word denoted by $\lambda $. Hence, $\Lambda ^{\ast }\left( I\right) $ is the
set of all finite words with letters from $I.$

For a family of functions $\left( f_{i}\right) _{i\in I},$ where $%
f_{i}:X\rightarrow X$ and $\omega =\omega _{1}\omega _{2}...\omega _{n}\in
\Lambda _{n}\left( I\right) $, we use the following notation: $f_{\omega
}=f_{\omega _{1}}\circ ...\circ f_{\omega _{n}}.$

For a set $I$, by $\Lambda \left( I\right) $ we mean the set $I^{%
\mathbb{N}
^{\ast }}$. The elements of $\Lambda \left( I\right) $ can be written as
infinite words, namely $\omega =\omega _{1}\omega _{2}...\omega _{n}...$ .
For $\omega \in \Lambda \left( I\right) $ and $n\in
\mathbb{N}
^{\ast }$, by $[\omega ]_{n}$ we mean the word formed with the first $n$
letters of $\omega $. For $\omega \in \Lambda _{m}\left( I\right) $ and $%
n\in
\mathbb{N}
^{\ast }$, by $[\omega ]_{n}$ we mean the word formed with the first $n$
letters of $\omega $ if $m\geq n$, or the word $\omega $ if $m\leq n$.

For a fixed $c\in \lbrack 0,1)$, we define a function $d_{c}:\Lambda
(I)\times \Lambda (I)\rightarrow \lbrack 0,\infty )$ by
\begin{equation*}
d_{c}\left( \alpha ,\beta \right) =\sum_{n\geq 1}c^{n}\left( 1-\delta
_{\alpha _{n}}^{\beta _{n}}\right) ,
\end{equation*}%
for all $\alpha ,\beta \in \Lambda (I)$, where by $\alpha _{n}$ we mean the
letter on position $n$ in $\alpha $ and $\delta _{\alpha _{n}}^{\beta
_{n}}=\left\{
\begin{array}{c}
1\text{, if }\alpha _{n}=\beta _{n} \\
0\text{, if }\alpha _{n}\neq \beta _{n}%
\end{array}%
\right. $ for all $n\in
\mathbb{N}
^{\ast }$.

\begin{remark}
\label{compact}$\left( \Lambda (I),d_{c}\right) $ is a complete metric
space. If $I$ is finite, then $\left( \Lambda (I),d_{c}\right) $ is compact.
\end{remark}

\bigskip

\textbf{Fuzzy sets }

\bigskip

\begin{definition}
Let $u:X\rightarrow
\mathbb{R}
$. The function $u$ is upper semicontinuous (usc) if, for each $\alpha \in
\mathbb{R}
$, the set $u^{-1}\left( [\alpha ,\infty )\right) :=\left\{ x\in X\text{ }|%
\text{ }u\left( x\right) \geq \alpha \right\} $ is closed.
\end{definition}

\begin{remark}
The function $u:X\rightarrow
\mathbb{R}
$ is usc if and only if%
\begin{equation*}
\overline{\lim_{x\rightarrow x_{0}}}u\left( x\right) \leq u\left(
x_{0}\right)
\end{equation*}%
for all $x_{0}\in X$.
\end{remark}

\begin{lemma}
\label{lema1}[see \cite{Str}] Let $\left( X,d\right) $ be a metric space and
$u:X\rightarrow
\mathbb{R}
$ an usc function. Then the restriction of $u$ to every set $K\in
P_{cp}\left( X\right) $ is a superior bounded function. Moreover, there
exists $x_{0}\in K$ such that $u\left( x_{0}\right) =\sup_{x\in K}u\left(
x\right) $.
\end{lemma}

\begin{definition}
Let $u:X\rightarrow \lbrack 0,1]$. We say that $u$ is a fuzzy subset of $X$.
The family of fuzzy subsets of $X$ is denoted by $\mathcal{F}_{X}$.
\end{definition}

\begin{definition}
Given $\alpha \in (0,1]$ and $u\in \mathcal{F}_{X}$, the grey level or $%
\alpha $-cut of $u$ is the set
\begin{equation*}
\left[ u\right] ^{\alpha }:=\left\{ x\in X\text{ }|\text{ }u\left( x\right)
\geq \alpha \right\} .
\end{equation*}%
For $\alpha =0$ we define
\begin{equation*}
\left[ u\right] ^{0}:=\text{supp}\left( u\right) :=\overline{\cup _{\alpha
\in (0,1]}[u]^{\alpha }}=\overline{\left\{ x\in X\text{ }|\text{ }u\left(
x\right) >0\right\} }.
\end{equation*}
\end{definition}

\begin{definition}
A fuzzy set $u\in \mathcal{F}_{X}$ is

a) normal, if there is $x\in X$ such that $u\left( x\right) =1$;

b) compactly supported if $\left[ u\right] ^{0}$ is compact.
\end{definition}

\begin{notation}
\begin{equation*}
\mathcal{F}_{X}^{\ast \ast }=\left\{ u\in \mathcal{F}_{X}\text{ }|\text{ }u%
\text{ is normal and compactly supported}\right\} \text{;}
\end{equation*}%
\begin{equation*}
\mathcal{F}_{X}^{\ast }=\left\{ u\in \mathcal{F}_{X}^{\ast \ast }\text{ }|%
\text{ }u\text{ is usc}\right\} .
\end{equation*}
\end{notation}

\begin{definition}
(Zadeh's Extension Principle) Let $\left( X,d\right) $ and $\left( Y,\rho
\right) $ be two metric spaces. Given a map $T:X\rightarrow Y$ and $u\in
\mathcal{F}_{X}$, we define a new fuzzy set $T\left( u\right) \in \mathcal{F}%
_{Y}$ as follows:%
\begin{equation*}
T\left( u\right) :Y\rightarrow \lbrack 0,1]
\end{equation*}%
\begin{equation*}
T\left( u\right) \left( y\right) :=\left\{
\begin{array}{c}
\sup_{T\left( x\right) =y}u\left( x\right) ,\text{ if }y\in T\left( X\right)
\\
0,\text{ otherwise}%
\end{array}%
\right.
\end{equation*}%
for all $y\in Y$.
\end{definition}

\begin{proposition}
\label{T(u)}[see \cite{Str}] Let $\left( X,d\right) $ and $\left( Y,\rho
\right) $ be two metric spaces and $T:X\rightarrow Y$ a continuous map. Let $%
u\in \mathcal{F}_{X}$. We have

a) If $u$ is normal, then $T\left( u\right) $ is normal;

b) If $u$ is compactly supported, then $T\left( u\right) $ is compactly
supported;

c) If $u$ is usc and compactly supported, then $T\left( u\right) $ is usc
and compactly supported.
\end{proposition}

\begin{remark}
Using the above Proposition, we deduce that if $u\in \mathcal{F}_{X}^{\ast }$%
\ (or $u\in \mathcal{F}_{X}^{\ast \ast }$), then $T\left( u\right) \in
\mathcal{F}_{X}^{\ast }$ (or $T\left( u\right) \in \mathcal{F}_{X}^{\ast
\ast }$) for all continuous functions $T:X\rightarrow Y$.
\end{remark}

\begin{proposition}
\lbrack see \cite{Str}] If $u\in \mathcal{F}_{X}^{\ast }$, then for every $%
\alpha \in \left[ 0,1\right] $, the $\alpha $-cut set of $u$ is nonempty and
compact.
\end{proposition}

The topology on $\mathcal{F}_{X}^{\ast \ast }$ is defined using the
Hausdorff-Pompeiu semidistance between the $\alpha $-cuts. Since $%
P_{b}\left( X\right) $ contains all the $\alpha $-cuts, we can define a
semidistance $d_{\infty }$ in $\mathcal{F}_{X}^{\ast \ast }$ by
\begin{equation*}
d_{\infty }\left( u,v\right) =\sup_{\alpha \in \left[ 0,1\right] }h\left( %
\left[ u\right] ^{\alpha },\left[ v\right] ^{\alpha }\right)
\end{equation*}%
for every $u,v\in \mathcal{F}_{X}^{\ast \ast }$. The restriction of $%
d_{\infty }$ to $\mathcal{F}_{X}^{\ast }$ is a metric, since in this case
the $\alpha $-cuts belong to $P_{cp}\left( X\right) $.

\begin{theorem}
\label{complete}[see \cite{Str}] $\ \left( \mathcal{F}_{X}^{\ast },d_{\infty
}\right) $ is a complete metric space provided $\left( X,d\right) $ is
complete.
\end{theorem}

\begin{remark}
A particular case for the below Lemma (when $u,v\in \mathcal{F}_{X}^{\ast }$%
) can be found in \cite{Str}. Here, we give a simpler proof then the one
presented in \cite{Str}.
\end{remark}

\begin{lemma}
\label{ldinf}
\begin{equation*}
d_{\infty }\left( u,v\right) =\sup_{\alpha \in (0,1]}h\left( \left[ u\right]
^{\alpha },\left[ v\right] ^{\alpha }\right)
\end{equation*}%
for every $u,v\in \mathcal{F}_{X}^{\ast \ast }$.
\end{lemma}

\begin{proof}
Let $u,v\in \mathcal{F}_{X}^{\ast \ast }$. Using the definiton of $d_{\infty
}$, we obtain
\begin{equation*}
d_{\infty }\left( u,v\right) \geq \sup_{\alpha \in (0,1]}h\left( \left[ u%
\right] ^{\alpha },\left[ v\right] ^{\alpha }\right) \text{.}
\end{equation*}

From the definition of $\left[ u\right] ^{0}$ we deduce
\begin{equation*}
\left[ u\right] ^{0}=\overline{\cup _{\alpha \in (0,1]}[u]^{\alpha }}\text{.}
\end{equation*}
\

Thus,%
\begin{equation*}
h\left( \left[ u\right] ^{0},\left[ v\right] ^{0}\right) =h\left( \overline{%
\cup _{\alpha \in (0,1]}[u]^{\alpha }},\overline{\cup _{\alpha \in
(0,1]}[v]^{\alpha }}\right)
\end{equation*}%
\begin{equation*}
=h\left( \cup _{\alpha \in (0,1]}[u]^{\alpha },\cup _{\alpha \in
(0,1]}[v]^{\alpha }\right) \overset{(\ref{sup})}{\leq }\sup_{\alpha \in
(0,1]}h\left( \left[ u\right] ^{\alpha },\left[ v\right] ^{\alpha }\right)
\text{.}
\end{equation*}%
\ \

It results
\begin{equation*}
d_{\infty }\left( u,v\right) =\sup_{\alpha \in \left[ 0,1\right] }h\left( %
\left[ u\right] ^{\alpha },\left[ v\right] ^{\alpha }\right) \leq
\sup_{\alpha \in (0,1]}h\left( \left[ u\right] ^{\alpha },\left[ v\right]
^{\alpha }\right)
\end{equation*}%
and from the both inequalities we obtain the conclusion.
\end{proof}

\bigskip

\textbf{Fuzzy iterated function systems }

\bigskip

Given a metric space $\left( X,d\right) $ and $\left( f_{i}\right) _{i\in I}$
a family of functions with $f_{i}:X\rightarrow
\mathbb{R}
$ for all $i\in I$, we denote by $\vee _{i\in I}f_{i}:X\rightarrow
\mathbb{R}
$ the function given by
\begin{equation*}
\left( \vee _{i\in I}f_{i}\right) \left( x\right) =\sup_{i\in I}f_{i}\left(
x\right)
\end{equation*}%
for every $x\in X$.

\begin{definition}
A grey level map is a nonzero function $\rho :\left[ 0,1\right] \rightarrow %
\left[ 0,1\right] $. We say that a nonzero map satisfies the non-decreasing
right continuous (ndrc) condition or it is a ndrc map if $\rho $ is
increasing and right continuous.
\end{definition}

\begin{definition}
Let $u\in \mathcal{F}_{X}^{\ast \ast }$ and $\rho :\left[ 0,1\right]
\rightarrow \left[ 0,1\right] $. The function $\rho \left( u\right)
:X\rightarrow \left[ 0,1\right] $ is given by $\rho \left( u\right) \left(
x\right) =\rho \left( u\left( x\right) \right) $ for all $x\in X$.
\end{definition}

\bigskip

From Proposition \ref{T(u)}, as in \cite{Str}, one can obtain:

\begin{proposition}
\lbrack see \cite{Str}] If $\rho :\left[ 0,1\right] \rightarrow \left[ 0,1%
\right] $ is a ndrc map, then for every $u\in \mathcal{F}_{X}^{\ast \ast }$,
the fuzzy set $\rho \left( u\right) \in \mathcal{F}_{X}^{\ast \ast }$. If $%
u\in \mathcal{F}_{X}^{\ast }$, then $\rho \left( u\right) \in \mathcal{F}%
_{X}^{\ast }$.
\end{proposition}

\begin{lemma}
\label{lema2}[see \cite{Str}] Let $J$ be a finite set and let $\left(
u_{j}\right) _{j\in J}$ be a family of usc functions, with $%
u_{j}:X\rightarrow \left[ 0,1\right] $ for all $j\in J$. Then $\vee _{j\in
J}u_{j}$ is usc.
\end{lemma}

\begin{remark}
Let us note that the above mentioned result is not true, in general, if the
set $J$ is infinite. For this reason, besides $\mathcal{F}_{X}^{\ast }$, we
had to use the set $\mathcal{F}_{X}^{\ast \ast }$ (see, for example, Lemmas %
\ref{lema3}, \ref{lema3bis}, \ref{distinf}, \ref{lemadiam}), and in
particular cases, we had to prove that a function which is given as $\vee
_{j\in J}u_{j}$ is usc, where the functions $u_{j}$, $j\in J$ are usc (see
the proofs of the Theorem \ref{thm}).
\end{remark}

\begin{definition}
A system of grey level maps $\left( \rho _{i}\right) _{i\in I}$ is
admissible if it satisfies the following conditions:

a) $\rho _{i}$ is a ndrc map, for all $i\in I$;

b) $\rho _{i}\left( 0\right) =0$ for all $i\in I$;

c) $\rho _{i}\left( 1\right) =1$ for some $i\in I$.
\end{definition}

\begin{definition}
\label{dfuzzy}Let $\mathcal{S}=\left( \left( X,d\right) ,\left( f_{i}\right)
_{i\in I}\right) $ be an orbital contractive iterated function system and
let $\left( \rho _{i}\right) _{i\in I}$ be an admissible system of grey
level maps.

The system $\mathcal{S}_{Z}=\left( \left( X,d\right) ,\left( f_{i}\right)
_{i\in I},\left( \rho _{i}\right) _{i\in I}\right) $ is called an orbital
fuzzy iterated function system. We define the fuzzy Hutchinson Barnsley
operator associated to $\mathcal{S}_{Z}$ by%
\begin{equation*}
Z:\mathcal{F}_{X}^{\ast \ast }\rightarrow \mathcal{F}_{X}^{\ast \ast }
\end{equation*}%
\begin{equation*}
Z\left( u\right) :=\vee _{i\in I}\rho _{i}\left( f_{i}\left( u\right) \right)
\end{equation*}%
for all $u\in $ $\mathcal{F}_{X}^{\ast \ast }$.
\end{definition}

\begin{proposition}
\lbrack see \cite{Str}]\label{PropZ(u)} Let $\mathcal{S}_{Z}=\left( \left(
X,d\right) ,\left( f_{i}\right) _{i\in I},\left( \rho _{i}\right) _{i\in
I}\right) $ be an orbital fuzzy iterated function system. Then $Z\left(
u\right) \in \mathcal{F}_{X}^{\ast }$ for all $u\in \mathcal{F}_{X}^{\ast }$
and $Z\left( u\right) \in \mathcal{F}_{X}^{\ast \ast }$ for all $u\in
\mathcal{F}_{X}^{\ast \ast }$.
\end{proposition}

\begin{remark}
Let $\mathcal{S}_{Z}=\left( \left( X,d\right) ,\left( f_{i}\right) _{i\in
I},\left( \rho _{i}\right) _{i\in I}\right) $ be an orbital fuzzy iterated
function system. If we suppose that the function $\rho _{i}:[0,1]\rightarrow
\lbrack 0,1]$ is continuous for every $i\in I$, then we obtain simple
justifications for Lemma \ref{Z} and Theorem \ref{thm}.
\end{remark}

\section{Main results}

Let $\mathcal{S}_{Z}=\left( \left( X,d\right) ,\left( f_{i}\right) _{i\in
I},\left( \rho _{i}\right) _{i\in I}\right) $ be an orbital fuzzy iterated
function system. We consider $\mathcal{F}_{\mathcal{S}}^{\ast \ast }=\{u\in
\mathcal{F}_{X}^{\ast \ast }$ $|$ if there exists $x\in X$ with the property
$u\left( x\right) >0$, then there exist $w\in X$ and $y\in \overline{%
\mathcal{O}\left( w\right) }$ such that $x\in \overline{\mathcal{O}\left(
w\right) }$ and $u\left( y\right) =1\}$ and $\mathcal{F}_{\mathcal{S}}^{\ast
}=\{u\in \mathcal{F}_{S}^{\ast \ast }$ $|$ $u$ is usc$\}$. The following
Lemma states that $\mathcal{F}_{\mathcal{S}}^{\ast }$ is a closed subset of $%
\mathcal{F}_{X}^{\ast }.$

\begin{lemma}
\label{Fs_closed}Let $\mathcal{S}_{Z}=\left( \left( X,d\right) ,\left(
f_{i}\right) _{i\in I},\left( \rho _{i}\right) _{i\in I}\right) $ be an
orbital fuzzy iterated function system, with $I$ an nonempty finite set.
Then $\left( \mathcal{F}_{\mathcal{S}}^{\ast },d_{\infty }\right) $ is a
closed subset of $\left( \mathcal{F}_{X}^{\ast },d_{\infty }\right) $.
\end{lemma}

\begin{proof}
Let $\left( u_{n}\right) _{n}$ be a sequence of elements from $\mathcal{F}_{%
\mathcal{S}}^{\ast }$ and $u\in \mathcal{F}_{X}^{\ast }$ such that $%
\lim_{n\rightarrow \infty }d_{\infty }\left( u_{n},u\right) =0$. As supp$u$
and supp$u_{n}$ are compact sets for all $n\in
\mathbb{N}
$ and $\lim_{n\rightarrow \infty }h\left( \text{supp}u_{n}\text{, supp}%
u\right) =0$, we deduce $K:=\left( \cup _{n\in
\mathbb{N}
}\text{supp}u_{n}\right) \cup $ supp$u$ is a compact set.

Let us consider $x\in X$ such that $u\left( x\right) >0$. Thus, $x\in K$.
Let $\alpha =u\left( x\right) $. As $\lim_{n\rightarrow \infty }d_{\infty
}\left( u_{n},u\right) =0$, we deduce $\lim_{n\rightarrow \infty }h\left( %
\left[ u_{n}\right] ^{\alpha },\left[ u\right] ^{\alpha }\right) =0$. Hence,
there exists a sequence $\left( x_{n}\right) _{n}$ such that $x_{n}\in \left[
u_{n}\right] ^{\alpha }\subset $supp$u_{n}$ for every $n\in
\mathbb{N}
$ and $\lim_{n\rightarrow \infty }d\left( x_{n},x\right) =0$.

As $u_{n}\in \mathcal{F}_{\mathcal{S}}^{\ast }$ for all $n\in
\mathbb{N}
$, we have that there exist $w_{n}\in K$ and $y_{n}\in \overline{\mathcal{O}%
\left( w_{n}\right) }$ such that $x_{n}\in \overline{\mathcal{O}\left(
w_{n}\right) }$ and $u_{n}\left( y_{n}\right) =1$. We obtain a sequence $%
\left( w_{n}\right) _{n}\subset K$ and by passing to a subsequence (which
will still be denoted by $\left( w_{n}\right) _{n}$) we deduce there exists $%
w\in K$ such that $\lim_{n\rightarrow \infty }d\left( w_{n},w\right) =0$. As
$y_{n}\in \overline{\mathcal{O}\left( w_{n}\right) }$ $\subset K$ for all $%
n\in
\mathbb{N}
,$ by passing to a subsequence (which will still be denoted by $\left(
y_{n}\right) _{n}$) we deduce there exists $y\in K$ such that $%
\lim_{n\rightarrow \infty }d\left( y_{n},y\right) =0$.

Therefore, we obtained the sequences $\left( w_{n}\right) _{n},\left(
y_{n}\right) _{n}\subset K$ such that $\left( w_{n}\right) _{n}$ is
convergent to $w\in K$ and $\left( y_{n}\right) _{n}$ is convergent to $y\in
K$. As $y_{n}\in \overline{\mathcal{O}\left( w_{n}\right) }$ for every $n\in
\mathbb{N}
$, by applying Lemma \ref{convpeorb}, we deduce $y\in \overline{\mathcal{O}%
\left( w\right) }$. We have $y_{n}\in \left[ u_{n}\right] ^{1}$ for all $%
n\in
\mathbb{N}
$; $\lim_{n\rightarrow \infty }h\left( \left[ u_{n}\right] ^{1},\left[ u%
\right] ^{1}\right) =0$ and $\lim_{n\rightarrow \infty }d\left(
y_{n},y\right) =0$. Hence, $y\in \left[ u\right] ^{1}$, so $u\left( y\right)
=1$.

In conclusion, we proved that $u\in \mathcal{F}_{\mathcal{S}}^{\ast }$.
\end{proof}

\begin{remark}
If $\left( X,d\right) $ is a complete metric space, then $\left( \mathcal{F}%
_{\mathcal{S}}^{\ast },d_{\infty }\right) $ is complete.
\end{remark}

\begin{lemma}
\label{lema3}Let $\mathcal{S}_{Z}=\left( \left( X,d\right) ,\left(
f_{i}\right) _{i\in I},\left( \rho _{i}\right) _{i\in I}\right) $ be an
orbital fuzzy iterated function system. Let $\left( u_{j}\right) _{j\in J}$
be a family (finite or infinite) of functions from $\mathcal{F}_{X}^{\ast
\ast }$ having the property that there exists a set $K\in P_{cp}\left(
X\right) $ such that supp$u_{j}\subset K$ for all $j\in J$. Let $%
f:X\rightarrow X$ be a continuous function. Then

1) $u:=\vee _{j\in J}u_{j}\in \mathcal{F}_{X}^{\ast \ast }$;

2) $f\left( \vee _{j\in J}u_{j}\right) =\vee _{j\in J}f\left( u_{j}\right) $;

3) if $u_{j}\in \mathcal{F}_{\mathcal{S}}^{\ast \ast }$ for all $j\in J$,
then $u\in \mathcal{F}_{\mathcal{S}}^{\ast \ast }$.
\end{lemma}

\begin{proof}
1) We note that supp$u\subset K\in P_{cp}\left( X\right) $. Thus, supp$u$ is
compact. Let $j_{0}\in J$ be fixed. Since $u_{j_{0}}\in \mathcal{F}%
_{X}^{\ast \ast }$, it results that there exists $x_{0}\in X$ such that $%
u_{j_{0}}\left( x_{0}\right) =1$. Hence,
\begin{equation*}
1\geq u\left( x_{0}\right) =\left( \vee _{j\in J}u_{j}\right) \left(
x_{0}\right) \geq u_{j_{0}}\left( x_{0}\right) =1
\end{equation*}%
and we deduce $u\in \mathcal{F}_{X}^{\ast \ast }$.

2) If $y\notin f\left( X\right) ,$ then
\begin{equation*}
f\left( \vee _{j\in J}u_{j}\right) \left( y\right) =0
\end{equation*}%
and%
\begin{equation*}
\left( \vee _{j\in J}f\left( u_{j}\right) \right) \left( y\right)
=\sup_{j\in J}f\left( u_{j}\right) \left( y\right) =0\text{.}
\end{equation*}

Hence, the relation is proved.

Let us consider $y\in f\left( X\right) $. We have%
\begin{equation*}
f\left( \vee _{j\in J}u_{j}\right) \left( y\right) =\sup_{f\left( x\right)
=y}\left( \vee _{j\in J}u_{j}\right) \left( x\right) =\sup_{f\left( x\right)
=y}\left( \sup_{j\in J}u_{j}\left( x\right) \right)
\end{equation*}%
\begin{equation*}
=\sup_{j\in J}\sup_{f\left( x\right) =y}u_{j}\left( x\right) =\sup_{j\in
J}f\left( u_{j}\right) \left( y\right) =\left( \vee _{j\in J}f\left(
u_{j}\right) \right) \left( y\right) \text{.}
\end{equation*}

Therefore, $f\left( \vee _{j\in J}u_{j}\right) \left( y\right) =\left( \vee
_{j\in J}f\left( u_{j}\right) \right) \left( y\right) $ for every $y\in
f\left( X\right) $ and we deduce the conclusion.

3) Let us consider $x\in X$ such that $u\left( x\right) >0$. Then, there
exists $j_{0}\in J$ such that $u_{j_{0}}\left( x\right) >0$. As $%
u_{j_{0}}\in \mathcal{F}_{\mathcal{S}}^{\ast \ast }$, there exist $w,y\in X$
such that $x,y\in \overline{\mathcal{O}\left( w\right) }$ and $%
u_{j_{0}}\left( y\right) =1$. We have $1=u_{j_{0}}\left( y\right) \leq \vee
_{j\in J}u_{j}\left( y\right) =u\left( y\right) \leq 1$, hence $u\left(
y\right) =1$.

Using the above conclusions, we deduce that $u\in \mathcal{F}_{\mathcal{S}%
}^{\ast \ast }$.
\end{proof}

\begin{lemma}
\label{lema3bis}Let $\mathcal{S}_{Z}=\left( \left( X,d\right) ,\left(
f_{i}\right) _{i\in I},\left( \rho _{i}\right) _{i\in I}\right) $ be an
orbital fuzzy iterated function system. If $u\in \mathcal{F}_{\mathcal{S}%
}^{\ast \ast }$, then $Z\left( u\right) \in \mathcal{F}_{\mathcal{S}}^{\ast
\ast }$. In particular, if $u\in \mathcal{F}_{\mathcal{S}}^{\ast }$, then $%
Z\left( u\right) \in \mathcal{F}_{\mathcal{S}}^{\ast }$.
\end{lemma}

\begin{proof}
Let us consider $u\in \mathcal{F}_{\mathcal{S}}^{\ast \ast }$ and $x\in X$
such that $Z\left( u\right) \left( x\right) >0$. Using the definition of $Z$%
, this is equivalent with $\max_{i\in I}\rho _{i}\left( f_{i}\left( u\right)
\left( x\right) \right) >0$. Then, there exists $i_{0}\in I$ such that $\rho
_{i_{0}}\left( f_{i_{0}}\left( u\right) \left( x\right) \right) >0$. As $%
\rho _{i_{0}}\left( 0\right) =0$, we deduce $f_{i_{0}}\left( u\right) \left(
x\right) >0$ and there exists $y_{x}\in X$ such that $f_{i_{0}}\left(
y_{x}\right) =x$ and $u\left( y_{x}\right) >0$.

Using the fact that $u\in \mathcal{F}_{\mathcal{S}}^{\ast \ast }$, we obtain
there exist $w_{x},z_{x}\in X$ such that $y_{x},z_{x}\in \overline{\mathcal{O%
}\left( w_{x}\right) }$ and $u\left( z_{x}\right) =1$.

We have%
\begin{equation*}
x=f_{i_{0}}\left( y_{x}\right) \in f_{i_{0}}\left( \overline{\mathcal{O}%
\left( w_{x}\right) }\right) \subset \overline{f_{i_{0}}\left( \mathcal{O}%
\left( w_{x}\right) \right) }\subset \overline{\mathcal{O}\left(
w_{x}\right) }
\end{equation*}%
and we obtain $x\in \overline{\mathcal{O}\left( w_{x}\right) }$.

We know that there exists $j_{0}\in I$ such that $\rho _{j_{0}}\left(
1\right) =1.$ We make the notation $z=f_{j_{0}}\left( z_{x}\right) \in
\overline{\mathcal{O}\left( w_{x}\right) }.$ Then,%
\begin{equation*}
1\geq f_{j_{0}}\left( u\right) \left( z\right) =f_{j_{0}}\left( u\right)
\left( f_{j_{0}}\left( z_{x}\right) \right) =\sup_{f_{j_{0}}\left( y\right)
=f_{j_{0}}\left( z_{x}\right) }u\left( y\right) \geq u\left( z_{x}\right) =1
\end{equation*}%
and we deduce $f_{j_{0}}\left( u\right) \left( z\right) =1$. Thus, $\rho
_{j_{0}}\left( f_{j_{0}}\left( u\right) \left( z\right) \right) =1$. As%
\begin{equation*}
1\geq Z\left( u\right) \left( z\right) \geq \rho _{j_{0}}\left(
f_{j_{0}}\left( u\right) \left( z\right) \right) =1,
\end{equation*}%
it results that $Z\left( u\right) \left( z\right) =1$ and therefore $Z\left(
u\right) \in \mathcal{F}_{\mathcal{S}}^{\ast \ast }$.

We know $\mathcal{F}_{\mathcal{S}}^{\ast }=\mathcal{F}_{\mathcal{S}}^{\ast
\ast }\cap \mathcal{F}_{X}^{\ast }$. If $u\in \mathcal{F}_{\mathcal{S}%
}^{\ast }$, then $u\in \mathcal{F}_{\mathcal{S}}^{\ast \ast }$ and $u\in
\mathcal{F}_{X}^{\ast }$. Then, $Z\left( u\right) \in \mathcal{F}_{\mathcal{S%
}}^{\ast \ast }$ and $Z\left( u\right) \in \mathcal{F}_{X}^{\ast }$. In
conclusion, $Z\left( u\right) \in \mathcal{F}_{\mathcal{S}}^{\ast }$.
\end{proof}

\begin{lemma}
\label{Z}Let $\mathcal{S}_{Z}=\left( \left( X,d\right) ,\left( f_{i}\right)
_{i\in I},\left( \rho _{i}\right) _{i\in I}\right) $ be an orbital fuzzy
iterated function system and let $\left( u_{j}\right) _{j\in J}$ be an
infinite family with elements from $\mathcal{F}_{X}^{\ast \ast }$ having the
following properties:

a) there exists $K\in P_{cp}\left( X\right) $ such that supp$u_{j}\subset K$
for all $j\in J$;

b) $\left( \vee _{j\in J}u_{j}\right) \left( x\right) =\left( \max_{j\in
J}u_{j}\right) \left( x\right) $ for all $x\in X$;

c) $\vee _{j\in J}u_{j}\in \mathcal{F}_{X}^{\ast }$.

Then,

1) $Z\left( \vee _{j\in J}u_{j}\right) $ $=\vee _{j\in J}Z\left(
u_{j}\right) \in \mathcal{F}_{X}^{\ast };$

2) $\left( \vee _{j\in J}Z\left( u_{j}\right) \right) \left( y\right)
=\max_{j\in J}\left( Z\left( u_{j}\right) \left( y\right) \right) $, for all
$y\in X$;

3) $Z^{n}\left( \vee _{j\in J}u_{j}\right) =\vee _{j\in J}Z^{n}\left(
u_{j}\right) \in \mathcal{F}_{X}^{\ast }$, for all $n\in
\mathbb{N}
$.
\end{lemma}

\begin{proof}
Let us fix $y\in X$ and $i\in I$.

We define the set $K_{i}:=\left\{ x\in K\text{ }|\text{ }f_{i}\left(
x\right) =y\right\} $, which is compact. Let us consider the function $%
g:K_{i}\rightarrow
\mathbb{R}
$ given by
\begin{equation*}
g\left( x\right) =\left( \vee _{j\in J}u_{j}\right) \left( x\right)
=\sup_{j\in J}u_{j}\left( x\right)
\end{equation*}%
for all $x\in K_{i}$. Using Lemma \ref{lema1} and c), we obtain there exists
$x_{0}\in K_{i}$ such that
\begin{equation*}
g\left( x_{0}\right) =\sup_{x\in K_{i}}g\left( x\right) =\sup_{f_{i}\left(
x\right) =y,x\in K_{i}}\left( \vee _{j\in J}u_{j}\right) \left( x\right)
=f_{i}\left( \vee _{j\in J}u_{j}\right) \left( y\right) .
\end{equation*}

Thus,
\begin{equation*}
g\left( x_{0}\right) =f_{i}\left( \vee _{j\in J}u_{j}\right) \left( y\right)
=\left( \vee _{j\in J}u_{j}\right) \left( x_{0}\right) \overset{\text{b)}}{=}%
\max_{j\in J}u_{j}\left( x_{0}\right) \text{.}
\end{equation*}

It implies that there exists $j_{0}\in I$ such that
\begin{equation}
f_{i}\left( \vee _{j\in J}u_{j}\right) \left( y\right) =u_{j_{0}}\left(
x_{0}\right) \text{.}  \label{1}
\end{equation}

Since $x_{0}\in K_{i}$, i.e. $f_{i}\left( x_{0}\right) =y$, we have%
\begin{equation*}
u_{j_{0}}\left( x_{0}\right) \leq \sup_{f_{i}\left( x\right)
=y}u_{j_{0}}\left( x\right) =f_{i}\left( u_{j_{0}}\right) \left( y\right)
\end{equation*}%
\begin{equation*}
\leq \sup_{f_{i}\left( x\right) =y}\sup_{j\in J}u_{j}\left( x\right)
=f_{i}\left( \vee _{j\in J}u_{j}\right) \left( y\right) \overset{(\ref{1})}{=%
}u_{j_{0}}\left( x_{0}\right) \text{.}
\end{equation*}

It results
\begin{equation*}
u_{j_{0}}\left( x_{0}\right) =f_{i}\left( u_{j_{0}}\right) \left( y\right)
=f_{i}\left( \vee _{j\in J}u_{j}\right) \left( y\right)
\end{equation*}%
\begin{equation*}
\overset{\text{Lemma \ref{lema3} 2)}}{=}\vee _{j\in J}f_{i}\left(
u_{j}\right) \left( y\right) =\sup_{j\in J}f_{i}\left( u_{j}\right) \left(
y\right) \text{.}
\end{equation*}

Therefore,%
\begin{equation}
f_{i}\left( u_{j_{0}}\right) \left( y\right) =\sup_{j\in J}f_{i}\left(
u_{j}\right) \left( y\right) =\max_{j\in J}f_{i}\left( u_{j}\right) \left(
y\right) .  \label{2}
\end{equation}

Using (\ref{2}) we have
\begin{equation*}
\rho _{i}\left( \vee _{j\in J}f_{i}\left( u_{j}\right) \left( y\right)
\right) =\rho _{i}\left( \max_{j\in J}f_{i}\left( u_{j}\right) \left(
y\right) \right)
\end{equation*}%
\begin{equation}
=\max_{j\in J}\rho _{i}\left( f_{i}\left( u_{j}\right) \left( y\right)
\right) =\max_{j\in J}\rho _{i}\left( f_{i}\left( u_{j}\right) \right)
\left( y\right) .  \label{3}
\end{equation}

Hence,%
\begin{equation*}
Z\left( \vee _{j\in J}u_{j}\right) \left( y\right) =\left( \vee _{i\in
I}\rho _{i}\left( f_{i}\left( \vee _{j\in J}u_{j}\right) \right) \right)
\left( y\right) =\max_{i\in I}\rho _{i}\left( f_{i}\left( \vee _{j\in
J}u_{j}\right) \right) \left( y\right)
\end{equation*}%
\begin{equation*}
\overset{\left( \ref{3}\right) }{=}\max_{i\in I}\max_{j\in J}\rho _{i}\left(
f_{i}\left( u_{j}\right) \right) \left( y\right) \overset{\text{Lemma \ref%
{lema_maxmax}}}{=}\max_{j\in J}\max_{i\in I}\rho _{i}\left( f_{i}\left(
u_{j}\right) \right) \left( y\right)
\end{equation*}%
\begin{equation*}
=\max_{j\in J}Z\left( u_{j}\right) \left( y\right) =\sup_{j\in J}Z\left(
u_{j}\right) \left( y\right) =\left( \vee _{j\in J}Z\left( u_{j}\right)
\right) \left( y\right) .
\end{equation*}

In conclusion, we obtained%
\begin{equation*}
Z\left( \vee _{j\in J}u_{j}\right) =\left( \vee _{j\in J}Z\left(
u_{j}\right) \right) \overset{\text{c), Prop. \ref{PropZ(u)} }}{\in }%
\mathcal{F}_{X}^{\ast }
\end{equation*}%
and
\begin{equation*}
\left( \vee _{j\in J}Z\left( u_{j}\right) \right) \left( y\right)
=\max_{j\in J}\left( Z\left( u_{j}\right) \left( y\right) \right) ,
\end{equation*}%
for all $y\in X.$

In order to prove 3), one can use 1), 2) and mathematical induction.
\end{proof}

\begin{lemma}
\label{distinf}Let $\left( u_{j}\right) _{j\in J}$ and $\left( v_{j}\right)
_{j\in J}$ be two infinite families with elements from $\mathcal{F}%
_{X}^{\ast \ast }$, having the following properties:

a) there exists $K\in P_{cp}\left( X\right) $ such that supp$u_{j}$, supp$%
v_{j}\subset K$ for all $j\in J$;

b) $\left( \vee _{j\in J}u_{j}\right) \left( x\right) =\left( \max_{j\in
J}u_{j}\right) \left( x\right) $ and $\left( \vee _{j\in J}v_{j}\right)
\left( x\right) =\left( \max_{j\in J}v_{j}\right) \left( x\right) $ for all $%
x\in X$.

Then,
\begin{equation*}
d_{\infty }\left( \vee _{j\in J}u_{j},\vee _{j\in J}v_{j}\right) \leq
\sup_{j\in J}d_{\infty }\left( u_{j},v_{j}\right) .
\end{equation*}
\end{lemma}

\begin{proof}
Let $\left( u_{j}\right) _{j\in J}$ and $\left( v_{j}\right) _{j\in J}$ be
two infinite families from $\mathcal{F}_{X}^{\ast \ast }$, which have the
properties a) and b). Applying Lemma \ref{ldinf}, we have%
\begin{equation*}
d_{\infty }\left( \vee _{j\in J}u_{j},\vee _{j\in J}v_{j}\right)
=\sup_{\alpha \in (0,1]}h\left( \left[ \vee _{j\in J}u_{j}\right] ^{\alpha },%
\left[ \vee _{j\in J}v_{j}\right] ^{\alpha }\right) .
\end{equation*}

Using the definition, we obtain%
\begin{equation*}
\left[ \vee _{j\in J}u_{j}\right] ^{\alpha }=\left\{ x\in X\text{ }|\text{ }%
\left( \vee _{j\in J}u_{j}\right) \left( x\right) \geq \alpha \right\}
\end{equation*}%
\begin{equation*}
\overset{\text{b)}}{=}\left\{ x\in X\text{ }|\text{ }\left( \max_{j\in
J}u_{j}\right) \left( x\right) \geq \alpha \right\}
\end{equation*}%
\begin{equation*}
=\cup _{j\in J}\left\{ x\in X\text{ }|\text{ }u_{j}\left( x\right) \geq
\alpha \right\} =\cup _{j\in J}\left[ u_{j}\right] ^{\alpha }.
\end{equation*}

Therefore,%
\begin{equation*}
d_{\infty }\left( \vee _{j\in J}u_{j},\vee _{j\in J}v_{j}\right)
=\sup_{\alpha \in (0,1]}h\left( \left[ \vee _{j\in J}u_{j}\right] ^{\alpha },%
\left[ \vee _{j\in J}v_{j}\right] ^{\alpha }\right)
\end{equation*}%
\begin{equation*}
=\sup_{\alpha \in (0,1]}h\left( \cup _{j\in J}\left[ u_{j}\right] ^{\alpha
},\cup _{j\in J}\left[ v_{j}\right] ^{\alpha }\right)
\end{equation*}%
\begin{equation}
\overset{(\ref{sup})}{\leq }\sup_{\alpha \in (0,1]}\sup_{j\in J}h\left( %
\left[ u_{j}\right] ^{\alpha },\left[ v_{j}\right] ^{\alpha }\right) \text{.}
\label{rel2}
\end{equation}

We have
\begin{equation}
h\left( \left[ u_{j}\right] ^{\alpha },\left[ v_{j}\right] ^{\alpha }\right)
\leq \sup_{\alpha \in \left[ 0,1\right] }h\left( \left[ u_{j}\right]
^{\alpha },\left[ v_{j}\right] ^{\alpha }\right) =d_{\infty }\left(
u_{j},v_{j}\right)  \label{rel1}
\end{equation}%
for all $j\in J$ and $\alpha \in \left[ 0,1\right] $. Using (\ref{rel2}) and
(\ref{rel1}) it results%
\begin{equation*}
d_{\infty }\left( \vee _{j\in J}u_{j},\vee _{j\in J}v_{j}\right) \leq
\sup_{\alpha \in (0,1]}\sup_{j\in J}h\left( \left[ u_{j}\right] ^{\alpha },%
\left[ v_{j}\right] ^{\alpha }\right)
\end{equation*}%
\begin{equation*}
\leq \sup_{\alpha \in (0,1]}\sup_{j\in J}d_{\infty }\left(
u_{j},v_{j}\right) =\sup_{j\in J}d_{\infty }\left( u_{j},v_{j}\right)
\end{equation*}%
and we obtain the conclusion.
\end{proof}

\begin{lemma}
\label{lemadiam}
\begin{equation*}
d_{\infty }\left( u,v\right) \leq diam\left( \left[ u\right] ^{0}\cup \left[
v\right] ^{0}\right)
\end{equation*}%
for all $u,v\in \mathcal{F}_{X}^{\ast \ast }$.
\end{lemma}

\begin{proof}
Let $u,v\in \mathcal{F}_{X}^{\ast \ast }$. We have%
\begin{equation*}
d_{\infty }\left( u,v\right) =\sup_{\alpha \in \left[ 0,1\right] }h\left( %
\left[ u\right] ^{\alpha },\left[ v\right] ^{\alpha }\right) \text{.}
\end{equation*}

Applying Proposition \ref{diam} and the fact that $\left[ u\right] ^{\alpha
}\subset \left[ u\right] ^{0}$ for all $\alpha \in \left[ 0,1\right] $, we
deduce%
\begin{equation*}
d_{\infty }\left( u,v\right) \leq \sup_{\alpha \in \left[ 0,1\right]
}diam\left( \left[ u\right] ^{\alpha }\cup \left[ v\right] ^{\alpha }\right)
\leq diam\left( \left[ u\right] ^{0}\cup \left[ v\right] ^{0}\right) \text{.}
\end{equation*}

Hence, we obtained
\begin{equation*}
d_{\infty }\left( u,v\right) \leq diam\left( \left[ u\right] ^{0}\cup \left[
v\right] ^{0}\right)
\end{equation*}%
for all $u,v\in \mathcal{F}_{X}^{\ast \ast }$.
\end{proof}

\begin{lemma}
\label{Z0}Let $\mathcal{S}_{Z}=\left( \left( X,d\right) ,\left( f_{i}\right)
_{i\in I},\left( \rho _{i}\right) _{i\in I}\right) $ be an orbital fuzzy
iterated function system. Let $F_{\mathcal{S}}$ be the fractal operator
associated to $\mathcal{S}_{Z}$ and let $Z$ be the fuzzy operator of $%
\mathcal{S}_{Z}$. Then,
\begin{equation*}
\left[ Z\left( u\right) \right] ^{0}\subset F_{\mathcal{S}}\left( \left[ u%
\right] ^{0}\right)
\end{equation*}%
for every $u\in \mathcal{F}_{\mathcal{S}}^{\ast \ast }$.
\end{lemma}

\begin{proof}
Let $\alpha \in (0,1]$ and $x\in \left[ Z\left( u\right) \right] ^{\alpha }$%
. Then, there exists $i_{0}\in I$ such that $\rho _{i_{0}}\left(
f_{i_{0}}\left( u\right) \right) \left( x\right) =\rho _{i_{0}}\left(
f_{i_{0}}\left( u\right) \left( x\right) \right) \geq \alpha $. As $\rho
_{i_{0}}\left( 0\right) =0$, we deduce $f_{i_{0}}\left( u\right) \left(
x\right) >0$. Let us consider $\beta >0$ such that
\begin{equation*}
f_{i_{0}}\left( u\right) \left( x\right) =\sup_{f_{i_{0}}\left( y\right)
=x}u\left( y\right) >\beta \text{.}
\end{equation*}%
Then, there exists $y_{x}\in X$ such that $u\left( y_{x}\right) >\beta $ and
$f_{i}\left( y_{x}\right) =x$. It implies that $y_{x}\in \left[ u\right]
^{\beta }$. Moreover, we have $f_{i_{0}}\left( y_{x}\right) =x$. Hence, $%
x\in f_{i_{0}}\left( \left[ u\right] ^{\beta }\right) \subset $ $%
f_{i_{0}}\left( \left[ u\right] ^{0}\right) \subset F_{\mathcal{S}}\left( %
\left[ u\right] ^{0}\right) $. Therefore, $\left[ Z\left( u\right) \right]
^{0}\subset F_{\mathcal{S}}\left( \left[ u\right] ^{0}\right) $.
\end{proof}

\begin{theorem}
\label{thm}The fuzzy operator associated to an orbital fuzzy iterated
function system $\mathcal{S}$ is a weakly Picard operator on $\mathcal{F}_{%
\mathcal{S}}^{\ast }$ $.$
\end{theorem}

\begin{proof}
Let $\mathcal{S}_{Z}=\left( \left( X,d\right) ,\left( f_{i}\right) _{i\in
I},\left( \rho _{i}\right) _{i\in I}\right) $ be such a system and let $Z$
be the fuzzy operator associated to $\mathcal{S}_{Z}$.

Let us consider $x\in X$ and $u\in $ $\mathcal{F}_{\mathcal{S}}^{\ast }$
such that $u\left( x\right) >0$. Thus, there exist $w_{x},y_{x}\in X$ such
that $x,y_{x}\in \overline{\mathcal{O}\left( w_{x}\right) }$ and $u\left(
y_{x}\right) =1$. We define the function $u^{x}:X\rightarrow \left[ 0,1%
\right] $ given by
\begin{equation*}
u^{x}\left( y\right) =\left\{
\begin{array}{c}
u\left( y\right) \text{, if }y\in \overline{\mathcal{O}\left( w_{x}\right) }
\\
0\text{, otherwise}%
\end{array}%
\right.
\end{equation*}%
for all $y\in X$. It is obvious that $u^{x}\in \mathcal{F}_{\mathcal{S}%
}^{\ast }$.

We make the following notation: $\left[ u\right] ^{\ast }=\cup _{0<\alpha
\leq 1}\left[ u\right] ^{\alpha }=\left\{ x\in X\text{ \ }|\text{ }u\left(
x\right) >0\right\} $.

\begin{claim}
\label{cl_ux}%
\begin{equation}
u=\vee _{x\in \left[ u\right] ^{\ast }}u^{x}  \label{ux}
\end{equation}%
for all $u\in \mathcal{F}_{\mathcal{S}}^{\ast }$.
\end{claim}

Justification: \

Let us consider $u\in \mathcal{F}_{\mathcal{S}}^{\ast }$ .

It is obvious that $u^{x}\leq u$ for all $x\in \left[ u\right] ^{\ast }$. It
follows that $\vee _{x\in \left[ u\right] ^{\ast }}u^{x}\leq u$.

Let $y\in X$. If $u\left( y\right) =0$, we have $u\left( y\right) \leq \vee
_{x\in \left[ u\right] ^{\ast }}u^{x}\left( y\right) $. If $u\left( y\right)
>0$, then $y\in \left[ u\right] ^{\ast }$ and $\vee _{x\in \left[ u\right]
^{\ast }}u^{x}\left( y\right) \geq u^{y}\left( y\right) =u\left( y\right) $.
Thus, $\vee _{x\in \left[ u\right] ^{\ast }}u^{x}\geq u$.

From the both inequalities we deduce the conclusion.

Let us note that in the above Claim we also proved that
\begin{equation}
\left( \vee _{x\in \left[ u\right] ^{\ast }}u^{x}\right) \left( y\right)
=u^{y}\left( y\right) =\max_{x\in \left[ u\right] ^{\ast }}u^{x}\left(
y\right)  \label{rem_max}
\end{equation}%
for all $y\in \left[ u\right] ^{\ast }$.

\begin{claim}
\label{Zn}%
\begin{equation*}
Z^{n}\left( u\right) =\vee _{x\in \left[ u\right] ^{\ast }}Z^{n}\left(
u^{x}\right)
\end{equation*}%
for all $u\in \mathcal{F}_{\mathcal{S}}^{\ast }$ and $n\in
\mathbb{N}
$.
\end{claim}

Justification:

Let $u\in \mathcal{F}_{\mathcal{S}}^{\ast }$. From relation (\ref{ux}) we
deduce%
\begin{equation*}
Z^{n}\left( u\right) =Z^{n}\left( \vee _{x\in \left[ u\right] ^{\ast
}}u^{x}\right)
\end{equation*}%
for all $n\in
\mathbb{N}
$. Using relation (\ref{rem_max}) and Lemma \ref{Z}, we obtain%
\begin{equation*}
Z^{n}\left( \vee _{x\in \left[ u\right] ^{\ast }}u^{x}\right) =\vee _{x\in %
\left[ u\right] ^{\ast }}Z^{n}\left( u^{x}\right)
\end{equation*}%
for all $n\in
\mathbb{N}
$. Therefore,%
\begin{equation*}
Z^{n}\left( u\right) =\vee _{x\in \left[ u\right] ^{\ast }}Z^{n}\left(
u^{x}\right)
\end{equation*}%
for all $u\in \mathcal{F}_{\mathcal{S}}^{\ast }$ and $n\in
\mathbb{N}
$. Thus, we completed the justification of the Claim.

Applying Claim \ref{Zn} and Lemma \ref{distinf} we deduce%
\begin{equation*}
d_{\infty }\left( Z^{n+1}\left( u\right) ,Z^{n}\left( u\right) \right)
=d_{\infty }\left( \vee _{x\in \left[ u\right] ^{\ast }}Z^{n+1}\left(
u^{x}\right) ,\vee _{x\in \left[ u\right] ^{\ast }}Z^{n}\left( u^{x}\right)
\right)
\end{equation*}%
\begin{equation*}
\leq \sup_{x\in \left[ u\right] ^{\ast }}d_{\infty }\left( Z^{n+1}\left(
u^{x}\right) ,Z^{n}\left( u^{x}\right) \right)
\end{equation*}%
for all $u\in \mathcal{F}_{\mathcal{S}}^{\ast }$ and $n\in
\mathbb{N}
$.

Using the fact that the system $\mathcal{S}_{Z}$ restricted to the $%
\overline{\mathcal{O}\left( w_{x}\right) }$ is a classical iterated function
system and applying Theorem 2.4.1 from \cite{Cabrelli}, we obtain%
\begin{equation*}
d_{\infty }\left( Z^{n+1}\left( u^{x}\right) ,Z^{n}\left( u^{x}\right)
\right) \leq C^{n}d_{\infty }\left( Z\left( u^{x}\right) ,u^{x}\right)
\end{equation*}%
for all $n\in
\mathbb{N}
$. Thus,%
\begin{equation}
d_{\infty }\left( Z^{n+1}\left( u\right) ,Z^{n}\left( u\right) \right) \leq
C^{n}\sup_{x\in \left[ u\right] ^{\ast }}d_{\infty }\left( Z\left(
u^{x}\right) ,u^{x}\right)  \label{r1}
\end{equation}%
for all $n\in
\mathbb{N}
.$

Now, let us consider $m,n\in
\mathbb{N}
$ with $m\leq n$. Using the triangle inequality and relation (\ref{r1}) we
have%
\begin{equation*}
d_{\infty }\left( Z^{m}\left( u\right) ,Z^{n}\left( u\right) \right) \leq
d_{\infty }\left( Z^{m}\left( u\right) ,Z^{m+1}\left( u\right) \right)
\end{equation*}%
\begin{equation*}
+d_{\infty }\left( Z^{m+1}\left( u\right) ,Z^{m+2}\left( u\right) \right)
+...+d_{\infty }\left( Z^{n-1}\left( u\right) ,Z^{n}\left( u\right) \right)
\end{equation*}%
\begin{equation*}
\leq \left( C^{m}+C^{m+1}+...+C^{n-1}\right) \cdot \sup_{x\in \left[ u\right]
^{\ast }}d_{\infty }\left( Z\left( u^{x}\right) ,u^{x}\right)
\end{equation*}%
\begin{equation*}
=\frac{C^{m}-C^{n}}{1-C}\cdot \sup_{x\in \left[ u\right] ^{\ast }}d_{\infty
}\left( Z\left( u^{x}\right) ,u^{x}\right) \text{.}
\end{equation*}

Hence,%
\begin{equation}
d_{\infty }\left( Z^{m}\left( u\right) ,Z^{n}\left( u\right) \right) \leq
\frac{C^{m}-C^{n}}{1-C}\cdot \sup_{x\in \left[ u\right] ^{\ast }}d_{\infty
}\left( Z\left( u^{x}\right) ,u^{x}\right)  \label{r2}
\end{equation}%
for all $m,n\in
\mathbb{N}
$ with $m\leq n$. Applying Lemma \ref{lemadiam} and then Lemma \ref{Z0} we
have%
\begin{equation*}
d_{\infty }\left( Z\left( u^{x}\right) ,u^{x}\right) \leq diam\left( \left[
Z\left( u^{x}\right) \right] ^{0}\cup \left[ u^{x}\right] ^{0}\right)
\end{equation*}%
\begin{equation}
\leq diam\left( F_{\mathcal{S}}\left( \left[ u^{x}\right] ^{0}\right) \cup %
\left[ u^{x}\right] ^{0}\right) \leq diam\left( F_{\mathcal{S}}\left( \text{%
supp}u\right) \cup \text{supp}u\right) .  \label{r3}
\end{equation}

Therefore, using (\ref{r2}) and (\ref{r3}) we obtain
\begin{equation}
d_{\infty }\left( Z^{m}\left( u\right) ,Z^{n}\left( u\right) \right) \leq
\frac{C^{m}-C^{n}}{1-C}\cdot diam\left( F_{\mathcal{S}}\left( \text{supp}%
u\right) \cup \text{supp}u\right)  \label{C1}
\end{equation}%
for all $m,n\in
\mathbb{N}
$ with $m\leq n$. Since $C\in \lbrack 0,1)$ and $diam\left( F_{\mathcal{S}%
}\left( \text{supp}u\right) \cup \text{supp}u\right) $ is finite, it results
that the sequence $\left( Z^{n}\left( u\right) \right) _{n\in
\mathbb{N}
}$ is Cauchy. Applying Theorem \ref{complete} we deduce $\left( Z^{n}\left(
u\right) \right) _{n\in
\mathbb{N}
}$ is convergent. We will denote its limit by $\mathbf{u}_{u}$. Using Lemma %
\ref{lema3bis}, we have $Z^{n}\left( u\right) \in \mathcal{F}_{\mathcal{S}%
}^{\ast }$. Applying Lemma \ref{Fs_closed}, we obtain $\mathbf{u}_{u}\in
\mathcal{F}_{\mathcal{S}}^{\ast }$. From relation (\ref{C1}), by passing to
limit as $n\rightarrow \infty $, we deduce%
\begin{equation}
d_{\infty }\left( Z^{m}\left( u\right) ,\mathbf{u}_{u}\right) \leq \frac{%
C^{m}}{1-C}\cdot diam\left( F_{\mathcal{S}}\left( \text{supp}u\right) \cup
\text{supp}u\right)  \label{C2}
\end{equation}%
for all $m\in
\mathbb{N}
$.

\begin{claim}
\label{cl1}%
\begin{equation*}
h\left( \left[ \rho _{i}\left( u\right) \right] ^{\alpha },\left[ \rho
_{i}\left( v\right) \right] ^{\alpha }\right) \leq d_{\infty }\left(
u,v\right)
\end{equation*}%
for all $u,v\in \mathcal{F}_{\mathcal{S}}^{\ast }$, $i\in I$ and $\alpha \in
(0,1]$.
\end{claim}

Justification: Let $u,v\in \mathcal{F}_{\mathcal{S}}^{\ast },$ $i\in I$ and $%
\alpha \in (0,1]$. We define the set $\beta _{\alpha ,i}:=\inf \left\{
\gamma \text{ }|\text{ }\rho _{i}\left( \gamma \right) \geq \alpha \right\} $%
. Thus, we deduce there exists a decreasing sequence $\left( \gamma
_{n}\right) _{n}$ which is convergent to $\beta _{\alpha ,i}$ and has the
property that $\rho _{i}\left( \gamma _{n}\right) \geq \alpha $ for all $%
n\in
\mathbb{N}
$ and $i\in I$. By passing to limit as $n\rightarrow \infty $ and using the
fact that $\rho _{i}$ is a right continuous function for all $i\in I$, we
deduce $\rho _{i}\left( \beta _{\alpha ,i}\right) \geq \alpha $ for all $%
i\in I$. Thus, $\left[ \rho _{i}\left( u\right) \right] ^{\alpha }=\left[ u%
\right] ^{\beta _{\alpha ,i}}$ for all $i\in I,$ $\alpha \in (0,1]$ and $%
u\in \mathcal{F}_{\mathcal{S}}^{\ast }$. We obtain%
\begin{equation*}
h\left( \left[ \rho _{i}\left( u\right) \right] ^{\alpha },\left[ \rho
_{i}\left( v\right) \right] ^{\alpha }\right) =h\left( \left[ u\right]
^{\beta _{\alpha ,i}},\left[ v\right] ^{\beta _{\alpha ,i}}\right) \leq
d_{\infty }\left( u,v\right)
\end{equation*}%
for all $i\in I$, $\alpha \in (0,1]$ and $u,v\in \mathcal{F}_{\mathcal{S}%
}^{\ast }$.

\begin{claim}
\label{cl_ineq}Let $\left( u_{n}\right) _{n}$ be a sequence of elements from
$\mathcal{F}_{\mathcal{S}}^{\ast }$ and let $u\in \mathcal{F}_{\mathcal{S}%
}^{\ast }$ such that $\lim_{n\rightarrow \infty }d_{\infty }\left(
u_{n},u\right) \rightarrow 0$. Then,
\begin{equation*}
d_{\infty }\left( Z\left( u_{n}\right) ,Z\left( u\right) \right) \leq
\max_{i\in I}d_{\infty }\left( f_{i}\left( u_{n}\right) ,f_{i}\left(
u\right) \right)
\end{equation*}%
for all $n\in
\mathbb{N}
$.
\end{claim}

Justification: Let $i\in I,$ $\alpha \in (0,1]$ and $n\in
\mathbb{N}
$. We have
\begin{equation*}
h\left( \left[ Z\left( u_{n}\right) \right] ^{\alpha },\left[ Z\left(
u\right) \right] ^{\alpha }\right) =h\left( \left[ \vee _{i\in I}\rho
_{i}\left( f_{i}\left( u_{n}\right) \right) \right] ^{\alpha },\left[ \vee
_{i\in I}\rho _{i}\left( f_{i}\left( u\right) \right) \right] ^{\alpha
}\right)
\end{equation*}%
\begin{equation*}
=h\left( \cup _{i\in I}\left[ \rho _{i}\left( f_{i}\left( u_{n}\right)
\right) \right] ^{\alpha },\cup _{i\in I}\left[ \rho _{i}\left( f_{i}\left(
u\right) \right) \right] ^{\alpha }\right)
\end{equation*}%
\begin{equation*}
\overset{(\ref{sup})}{\leq }\max_{i\in I}h\left( \left[ \rho _{i}\left(
f_{i}\left( u_{n}\right) \right) \right] ^{\alpha },\left[ \rho _{i}\left(
f_{i}\left( u\right) \right) \right] ^{\alpha }\right)
\end{equation*}%
\begin{equation*}
\overset{\text{Claim }\ref{cl1}}{\leq }\max_{i\in I}d_{\infty }\left(
f_{i}\left( u_{n}\right) ,f_{i}\left( u\right) \right)
\end{equation*}%
for all $\alpha \in (0,1]$ and $n\in
\mathbb{N}
$. Therefore,%
\begin{equation*}
d_{\infty }\left( Z\left( u_{n}\right) ,Z\left( u\right) \right) \leq
\max_{i\in I}d_{\infty }\left( f_{i}\left( u_{n}\right) ,f_{i}\left(
u\right) \right)
\end{equation*}%
for all $n\in
\mathbb{N}
$.

\begin{claim}
\label{cl_cont} The operator $Z$ is a continuous function.
\end{claim}

Justification: Let $\left( u_{n}\right) _{n}$ be a sequence of elements from
$\mathcal{F}_{\mathcal{S}}^{\ast }$ and let $u\in \mathcal{F}_{\mathcal{S}%
}^{\ast }$ such that $\lim_{n\rightarrow \infty }d_{\infty }\left(
u_{n},u\right) =0$. We deduce $\lim_{n\rightarrow \infty }h\left( \left[
u_{n}\right] ^{0},\left[ u\right] ^{0}\right) =0$. Applying Proposition \ref%
{h_cont} and taking into account that $\left[ u\right] ^{0}$ and $\left[
u_{n}\right] ^{0}$ are compact sets for all $n\in
\mathbb{N}
$, it results $K_{0}:=\left[ u\right] ^{0}\cup \left( \cup _{n\geq 1}\left[
u_{n}\right] ^{0}\right) \in P_{cp}\left( X\right) $.

Let $\varepsilon >0$. As the function $f_{i}$ is uniformly continuous on $%
K_{0}$ for all $i\in I$ and $I$ is finite, we deduce there exists $\delta
_{\varepsilon }>0$ such that
\begin{equation}
d\left( f_{i}\left( x\right) ,f_{i}\left( y\right) \right) <\varepsilon
\label{claim5_r1}
\end{equation}%
for all $i\in I$ and $x,y\in K_{0}$ with $d\left( x,y\right) <\delta
_{\varepsilon }$.

As $\lim_{n\rightarrow \infty }d_{\infty }\left( u_{n},u\right) =0$, we
obtain there exists $n_{\varepsilon }\in
\mathbb{N}
$ such that $d_{\infty }\left( u_{n},u\right) <\delta _{\varepsilon }$ for
every $n\in
\mathbb{N}
$, $n\geq n_{\varepsilon }$. Then,
\begin{equation}
h\left( \left[ u_{n}\right] ^{\alpha },\left[ u\right] ^{\alpha }\right)
\leq d_{\infty }\left( u_{n},u\right) <\delta _{\varepsilon }
\label{claim5_r2}
\end{equation}%
for all $\alpha \in (0,1]$ and $n\in
\mathbb{N}
$, $n\geq n_{\varepsilon }$.

Applying relations (\ref{claim5_r1}) and (\ref{claim5_r2}), we deduce
\begin{equation*}
h\left( \left[ f_{i}\left( u_{n}\right) \right] ^{\alpha },\left[
f_{i}\left( u\right) \right] ^{\alpha }\right) =h\left( f_{i}\left( \left[
u_{n}\right] ^{\alpha }\right) ,f_{i}\left( \left[ u\right] ^{\alpha
}\right) \right) <\varepsilon
\end{equation*}%
for all $i\in I$, $\alpha \in (0,1]$ and $n\in
\mathbb{N}
$, $n\geq n_{\varepsilon }$. It results
\begin{equation*}
d_{\infty }\left( f_{i}\left( u_{n}\right) ,f_{i}\left( u\right) \right)
\leq \varepsilon
\end{equation*}%
for all $i\in I$ and $n\in
\mathbb{N}
$, $n\geq n_{\varepsilon }$.

Taking into account Claim \ref{cl_ineq} and the above relation, it results
\begin{equation*}
d_{\infty }\left( Z\left( u_{n}\right) ,Z\left( u\right) \right) \leq
\max_{i\in I}d_{\infty }\left( f_{i}\left( u_{n}\right) ,f_{i}\left(
u\right) \right) \leq \varepsilon
\end{equation*}%
for all $n\in
\mathbb{N}
$, $n\geq n_{\varepsilon }$. In conclusion, $Z$ is continuous.

Using the fact that $\left( Z^{n}\left( u\right) \right) _{n\in
\mathbb{N}
}$ is convergent to $\mathbf{u}_{u}$ and applying Claim \ref{cl_cont}, we
deduce $Z\left( \mathbf{u}_{u}\right) =\mathbf{u}_{u}$. Therefore, $Z$ is a
weakly Picard operator.
\end{proof}

\section{Example}

\bigskip

Let us consider the metric space $\left(
\mathbb{R}
^{2},\left\Vert .\right\Vert _{2}\right) $, where $\left\Vert .\right\Vert
_{2}$ is the Euclidean norm, and the functions $f_{1}$ and $f_{2}$, with $%
f_{1},$ $f_{2}:%
\mathbb{R}
^{2}\rightarrow
\mathbb{R}
^{2}$, given by $f_{1}\left( x,y\right) =\left( x,\frac{1}{2}y\right) $ and $%
f_{2}\left( x,y\right) =\left( x,\frac{1}{2}y+\frac{1}{2}\right) $ for all $%
\left( x,y\right) \in
\mathbb{R}
^{2}$. Let us consider $I=\left\{ 1,2\right\} $. One can easily prove that $%
\mathcal{S}=\left( \left(
\mathbb{R}
^{2},\left\Vert .\right\Vert _{2}\right) ,\left( f_{i}\right) _{i\in
I}\right) $ is an orbital iterated function system and for every $K\in
P_{cp}\left(
\mathbb{R}
^{2}\right) ,$ the attractor is $A_{K}=\left[ 0,1\right] \times \pi
_{2}\left( K\right) $ (see \textbf{example A} from \cite{Savu}). Let us
consider the admissible system of functions $\left( \rho _{i}\right) _{i\in
I}$, given by $\rho _{1}\left( t\right) =t$ and $\rho _{2}\left( t\right) =%
\frac{3t}{4}$ for every $t\in \left[ 0,1\right] $. We consider the function $%
u\in \mathcal{F}_{\mathcal{S}_{\Lambda }}^{\ast }$, defined by
\begin{equation*}
u\left( x,y\right) =\left\{
\begin{array}{c}
1,\text{ if }\left( x,y\right) \in \lbrack 0,1]\times \left\{ 0\right\} \\
0\text{, otherwise}%
\end{array}%
,\right.
\end{equation*}
for all $\left( x,y\right) \in
\mathbb{R}
^{2}$.

Let $x\in \left[ 0,1\right] $. We consider the system restricted to $%
X_{x}:=\left\{ x\right\} \times
\mathbb{R}
$. In this case, the attractor is $A_{x}=\left\{ x\right\} \times \left[ 0,1%
\right] $. We want to find the function $\mathbf{u}_{x}$.

Let us consider $p:\Lambda ^{\ast }\left( I\right) \cup \Lambda \left(
I\right) \rightarrow \left[ 0,1\right] $, given by $p\left( \alpha \right)
=\sum_{n=1}^{|\alpha |}\frac{\alpha _{i}-1}{2^{n}},$for all $\alpha \in
\Lambda ^{\ast }\left( I\right) \cup \Lambda \left( I\right) $. We also
define the function $\eta :$ $\Lambda ^{\ast }\left( I\right) \cup \Lambda
\left( I\right) \rightarrow
\mathbb{N}
\cup \left\{ +\infty \right\} $, given by $\eta \left( \alpha \right)
=\left\vert \left\{ \alpha _{i}\text{ }|\text{ }\alpha _{i}=2,i\in \left\{
1,|\alpha |\right\} \right\} \right\vert $ for all $\alpha \in \Lambda
^{\ast }\left( I\right) \cup \Lambda \left( I\right) $, where for a set $A$
we denoted by $|A|$ the number of elements of $A$.

It is easy to see that $Z\left( u\right) \left( x,y\right) =\left\{
\begin{array}{c}
1\text{, if }x\in \lbrack 0,1]\text{ and }y=0 \\
\frac{3}{4}\text{, if }x\in \lbrack 0,1]\text{ and }y=\frac{1}{2} \\
0\text{, otherwise}%
\end{array}%
,\right. $ for all $\left( x,y\right) \in
\mathbb{R}
^{2}$. Using mathematical induction, one can prove that $Z^{n}\left(
u\right) \left( x,y\right) =$ $\max_{p\left( \alpha \right) =y,|\alpha |\leq
n}\left( \frac{3}{4}\right) ^{\eta \left( \alpha \right) }$, if $\left(
x,y\right) \in \left[ 0,1\right] ^{2}$ and there exists $\alpha \in \Lambda
^{\ast }\left( I\right) $ such that $|\alpha |\leq n$, $y=p\left( \alpha
\right) $ and $Z^{n}\left( u\right) \left( x,y\right) =0$, otherwise, for
all $\left( x,y\right) \in
\mathbb{R}
^{2}$. By passing to limit, we have $\mathbf{u}_{x}=\lim_{n\rightarrow
\infty }Z^{n}\left( u\right) $.

\bigskip

The following figures illustrate the example presented above:

\definecolor{C1}{cmyk}{0, 0.0, 0.0, 0.75}
\definecolor{C2}{cmyk}{0, 0.0, 0.0, 0.5}

\begin{figure}[H]
\begin{tikzpicture}[scale=1]
    \node at (-0.15, -0.2) {0};
    \draw [->] (0,-1) -- (0,1.5) node (yaxis) [above] {$y$};
    \draw [->] (-2,0) -- (5,0) node (xaxis) [right] {$x$};
  \draw [line width=0.4mm, black ] (0,0) -- (4,0);
\end{tikzpicture}
\caption{Step 1}
\end{figure}

\begin{figure}[H]
\begin{tikzpicture}
\node at (-0.15, -0.2) {0};
    \draw [->] (0,-1) -- (0,3) node (yaxis) [above] {$y$};
    \draw [->] (-2,0) -- (5,0) node (xaxis) [right] {$x$};

\draw [line width=0.4mm, black ] (0,0) -- (4,0);

\draw [line width=0.4mm, C1 ] (0,2) -- (4,2);

\end{tikzpicture}
\caption{Step 2}
\end{figure}

\begin{figure}[H]
\begin{tikzpicture}

\node at (-0.15, -0.2) {0};
    \draw [->] (0,-1) -- (0,4) node (yaxis) [above] {$y$};
    \draw [->] (-2,0) -- (5,0) node (xaxis) [right] {$x$};

\draw [line width=0.4mm, black ] (0,0) -- (4,0);

\draw [line width=0.4mm, C1 ] (0,2) -- (4,2);
\draw [line width=0.4mm, C1 ] (0,1) -- (4,1);

\draw [line width=0.4mm, C2 ] (0,3) -- (4,3);
\end{tikzpicture}
\caption{Step 3}
\end{figure}

\begin{figure}[H]
\begin{tikzpicture}

\node at (-0.15, -0.2) {0};
    \draw [->] (0,-1) -- (0,4.1) node (yaxis) [above] {$y$};
    \draw [->] (-2,0) -- (5,0) node (xaxis) [right] {$x$};

\draw [line width=0.4mm, black ] (0,0) -- (4,0);

\draw [line width=0.4mm, C1 ] (0,2) -- (4,2);
\draw [line width=0.4mm, C1 ] (0,1) -- (4,1);
\draw [line width=0.4mm, C1 ] (0,0.5) -- (4,0.5);
\draw [line width=0.4mm, C1 ] (0,0.25) -- (4,0.25);

\draw [line width=0.4mm, C2 ] (0,3) -- (4,3);
\draw [line width=0.4mm, C2 ] (0,2.5) -- (4,2.5);
\draw [line width=0.4mm, C2 ] (0,2.25) -- (4,2.25);
\draw [line width=0.4mm, C2 ] (0,1.5) -- (4,1.5);
\draw [line width=0.4mm, C2 ] (0,1.25) -- (4,1.25);
\draw [line width=0.4mm, C2 ] (0,0.75) -- (4,0.75);
\draw [line width=0.4mm, C2 ] (0,0.625) -- (4,0.625);

\end{tikzpicture}
\caption{Step 4}
\end{figure}

\bigskip

\textbf{Alexandru MIHAIL}

Faculty of Mathematics and Computer Science

Bucharest University, Romania

Str. Academiei, 010014, Bucharest, Romania

E-mail:\ mihail\_alex@yahoo.com

\bigskip

\textbf{Irina SAVU}

Faculty of Applied Sciences

Politehnica University of Bucharest, Romania

Splaiul Independen\c{t}ei 313, Bucharest, Romania

E-mail:\ maria\_irina.savu@upb.ro

\end{singlespace}
\end{document}